\documentclass[a4,12pt]{article}%

\usepackage{graphicx}
\usepackage{graphics}

\usepackage{amsmath}
\usepackage{amsfonts}
\usepackage{amssymb}
\usepackage{enumerate}
\usepackage{xypic}
\newtheorem{theorem}{Theorem}[section]

\newtheorem{definition}[theorem]{Definition}

\newtheorem{lemma}[theorem]{Lemma}

\newenvironment{proof}[1][Proof]{\textbf{#1.} }{\ \rule{0.5em}{0.5em}}

\newcommand{\rmd}{{\mathrm d}}

\newcommand{\dN}{{\bf N}}

\newcommand{\dR}{{\bf R}}

\newcommand{\ep}{\varepsilon}

\newcounter{figurecounter}
\begin{document}

\title{Logit Equilibrium as an Approximation of Nash Equilibrium%
\thanks{E. Solan acknowledges the support of the Israel Science Foundation, grant \#217/17.}}

\author{Eilon Solan and Omri N.~Solan%
\thanks{The School of Mathematical Sciences, Tel Aviv
University, Tel Aviv 6997800, Israel. e-mail: eilons@post.tau.ac.il (corresponding author), omrisola@post.tau.ac.il.}}

\maketitle

\begin{abstract}
We prove that the graph of the logit equilibrium correspondence is
a smooth manifold, which uniformly approximates the graph of the Nash equilibrium manifold.
\end{abstract}


\section{Introduction}

Kohlberg and Mertens (1986, Theorem~1) showed that the graph of the Nash equilibrium correspondence is homeomorphic to the set of payoff functions.
Ritzberger (1994, Proposition~2) proved that the graph of the Nash equilibrium correspondence can be uniformly approximated by a smooth manifold.
In this note we provide a specific smooth manifold that uniformly approximates the graph of the Nash equilibrium correspondence,
namely, the graph of the logit equilibrium correspondence, a solution concept that was defined by
McKelvey and Palfrey (1995).

The significance of this result stems from the need to apply topological results to the graph of the Nash equilibrium correspondence.
Various topological results are proven for smooth manifolds.
Since the graph of the Nash equilibrium correspondence is not a smooth manifold, these results cannot be applied to this graph.
The ability to apply these results to a smooth manifold that uniformly approximates the graph of the Nash equilibrium correspondence may be sufficient for various proofs.
Though by Ritzberger's (1994) result it is known that the graph of the Nash equilibrium correspondence can be approximated
by \emph{some} smooth manifold, the logit equilibrium has the additional advantage that it is completely mixed,
a property that can be sometimes useful, see, e.g, Solan and Solan (2018).

\section{The Model and Main Result}

A \emph{strategic game form} is a pair $(I,A)$
where $I = \{1,2,\ldots,d\}$ is a finite set of players and $A = \times_{i \in I} A_i$ is the Cartesian product of finite sets of \emph{pure strategies} for the players.
A \emph{payoff function for player~$i$} for the strategic game form $(I,A)$ is a function $u_i : A \to \dR$,
and a \emph{payoff function} is a collection $u = (u_i)_{i \in I}$ of payoff functions for the players.
Consequently, the set of all payoff functions is equivalent to $\dR^{A \times I}$.
A triplet $(I,A,u)$ where $u$ is a payoff function for the strategic game form $(I,A)$ is a \emph{game}.

A \emph{mixed strategy} for player~$i$ is a probability distribution $x_i \in\Delta(A_i)$,
and a \emph{mixed strategy profile} is a collection $x = (x_i)_{i \in I}$ of mixed strategies for the players.
It follows that the set of all mixed strategy profiles
%\footnote{When writing $\cup_{i \in I} A_i$ we implicitly assume that the sets of pure %strategies of the players are disjoint.}
is $X := \times_{i \in I} \Delta(A_i) \subset \dR^{\cup_{i \in I} A_i}$.
A payoff function $u_i$ for player~$i$ is extended to a function from $X$ to $\dR$ in a multilinear fashion.

A mixed strategy profile $x \in X$ is a (Nash) \emph{equilibrium} of the game $(I,A,u)$ if $u_i(x) \geq u_i(a_i,x_{-i})$
for every player $i \in I$ and every pure strategy $a_i \in A_i$.
When the strategic game form is fixed,
the graph of the Nash equilibrium correspondence is the collection of all pairs of a payoff function and equilibrium in the game induced by this payoff function.

\begin{definition}
Let $(I,A)$ be a strategic game form.
The \emph{graph of the Nash equilibrium correspondence} of $(I,A)$ is the set
\[ M := \left\{ (u,x) \in \dR^{A\times I} \times X \colon x \hbox{ is an equilibrium of } (I,A,u)\right\} \subset \dR^{A \times I} \times \dR^{\cup_{i \in I}A_i}. \]
\end{definition}

As mentioned above, Kohlberg and Mertens (1986) proved that the set $M$ is homeomorphic to the set of games, namely, to $\dR^{A \times I}$.
An important concept that we will need is that of logit equilibrium, which we define now.

\begin{definition}[McKelvey and Palfrey, 1995]
Let $(I,A,u)$ be a game and let $n > 0$.
The mixed strategy profile $x$ is a \emph{logit equilibrium with parameter $n$} of the game $(I,A,u)$ if for every player $i \in I$ and every pure strategy $a_i \in A_i$,
\begin{equation}
\label{equ:60}
x_i(a_i) = \frac{\exp(nu_i(x,a_i))}{\sum_{a'_i \in A_i}\exp(nu_i(x,a'_i))}.
\end{equation}
\end{definition}

Standard continuity arguments show that a limit of logit equilibria with parameter $n$ as $n$ goes to infinity is a Nash equilibrium, see McKlevey and Palfrey (1995, Theorem 2).

\begin{definition}
Let $(I,A)$ be a strategic game form.
For every real number $n$,
the \emph{graph of the logit equilibrium correspondence} of $(I,A)$ is the set
\[ M_n := \left\{ (u,x) \colon x \hbox{ is a logit equilibrium with parameter } n \hbox{ in } (I,A,u)\right\} \subset \dR^{A \times I} \times \dR^{\cup_{i \in I}A_i}. \]
\end{definition}

Our first main result is that the graph of the logit correspondence is a smooth manifold.
\begin{theorem}
\label{theorem:1}
The set $M_n$ is a smooth manifold of dimension $|A| \times |I|$.
\end{theorem}

Our second main result is that the graph of the logit correspondence uniformly approximates the graph of the Nash equilibrium correspondence.
\begin{theorem}
\label{theorem:2}
There are a function $\varphi : M \to \dR^{A \times I}$,
and for every $n \in \dN$ there is a smooth function $\varphi_n : M_n \to \dR^{A \times I}$
that satisfy the following property:
For every $\ep > 0$ there is $N = N(\ep) > 0$ such that
for every $n \geq N$ we have
\[ \|\varphi^{-1}(y) - (\varphi_n)^{-1}(y)\|_2 \leq \ep, \ \ \ \forall y \in \dR^{A \times I}. \]
\end{theorem}

\section{Proofs}

To prove that $M_n$ is a smooth manifold we need to study a certain function that will be used in the definition of
the immersion
between $M_n$ and $\dR^{A \times I}$.
Recall that an \emph{immersion} is a differentiable function between differentiable manifolds whose derivative is everywhere injective (one-to-one).
The keen reader will identify the origin of this function
and the proof of Theorem~\ref{theorem:1} below in the work of Kohlberg and Mertens (1986).

\begin{lemma}
\label{lemma:g}
For every $n > 0$ define the function $g^{(n)} : \dR^d \to \dR^d$ by
\[ g^{(n)}_i(x) = x_i + \frac{\exp(nx_i)}{\sum_{j=1}^d \exp(n x_j)}, \ \ \ \forall i\in \{1,2,\cdots,d\}. \]
The function $g^{(n)}$ is one-to-one, onto, and an immersion.
\end{lemma}

\begin{proof}

\noindent\textbf{Step 1:} \emph{The function $g^{(n)}$ is an immersion.}

An $n \times n$ matrix $A$ is a \emph{CL-matrix} if
(a) its diagonal entries are positive,
(b) its off-diagonal entries are negative,
and
(c) the sum of elements in each column is positive.
Thus, CL-matrices are subclasses of both L-matrices and column strictly diagonally dominant matrices.
By the Levy-Desplanques Theorem, every CL-matrix is invertible.

We first argue that the Jacobian matrix of $g^{(n)}$ is a CL-matrix at all points.
Indeed, simple algebraic calculations show that for every $i \in \{1,2,\cdots,d\}$,
\begin{eqnarray}
\frac{\partial g^{(n)}_i}{\partial x_i}(x) &=& 1 + \frac{n\exp(nx_i)\left( \sum_{k \neq i} \exp(nx_k)\right)}{\left(\sum_{k=1}^d \exp(nx_k)\right)^2} > 0,\\
\frac{\partial g^{(n)}_i}{\partial x_j}(x) &=& -\frac{n\exp(n (x_i+x_j))}{\left(\sum_{k=1}^d \exp(nx_k)\right)^2} < 0, \ \ \ \forall j \neq i.
\end{eqnarray}
In particular, Conditions~(a) and~(b) hold for the Jacobian matrix of $g^{(n)}$ at every point $x$.
We also have
\[ \sum_{i=1}^d g^{(n)}_{i}(x) = 1 + \sum_{i=1}^d x_i, \]
and therefore
\[ \sum_{i=1}^d\frac{\partial g^{(n)}_i}{\partial x_j}(x) = 1 > 0, \ \ \ \forall j \in \{1,2,\ldots,d\}, \]
so that Condition~(c) holds as well, and the Jacobian matrix is a CL-matrix at all points.
It follows that the Jacobian matrix is invertible at all points,
hence $g^{(n)}$ is an immersion.

\bigskip

\noindent\textbf{Step 2:} \emph{The function $g^{(n)}$ is onto.}

To prove that $g^{(n)}$ is onto we will show that its image is both open and closed.
Since the Jacobian matrix of $g^{(n)}$ at every point $x$ is invertible,
by the Open Mapping Theorem the image of $g^{(n)}$ is an open set.
To show that the image of $g^{(n)}$ is closed,
note that $\|x-g^{(n)}(x)\|_2 \leq 1$ for every $x \in \dR^d$,
and consider a sequence $(y^{k})_{k \in \dN}$ of points in the image of $g$ that converges to a point $y$.
For each $k \in \dN$ let $x^{k} \in \dR^d$ satisfy $y^{k} = g^{(n)}(x^{k})$.
Since $\| x^{k} - y^{k}\|_2 \leq 1$,
and since the sequence $(y^{k})_{k \in \dN}$ converges,
it follows that there is a subsequence $(x^{k_l})_{l \in \dN}$ that converges to a limit $x$.
Since the function $g^{(n)}$ is continuous, $g^{(n)}(x) = y$, so that $y$ is in the image of $g^{(n)}$,
which implies that the image of $g^{(n)}$ is closed.

\bigskip

\noindent\textbf{Step 3:} \emph{The function $g^{(n)}$ is one-to-one.}

We argue that any function whose Jacobian matrix is a CL-matrix is one-to-one.
Indeed, let $f$ be such a function, assume w.l.o.g.~that $f(\vec 0) = \vec 0$,
and fix $x \neq \vec 0$.
We will show that $f(x) \neq \vec 0$.
We have
\[ f(x) = f(0) + \int_{t=0}^1 df_{tx} \cdot x \rmd t = \left( \int_{t=0}^1 df_{tx} \rmd t \right) \cdot x. \]
The matrix $\int_{t=0}^1 df_{tx} \rmd t$, as an integral of CL-matrices,
is a CL-matrix, hence invertible.
In particular, $f(x) = \left( \int_{t=0}^1 df_{tx} \rmd t \right) \cdot x \neq \vec 0$,
as claimed.
\end{proof}

We are now ready to prove Theorem~\ref{theorem:1}.

\bigskip

\begin{proof}[Proof of Theorem~\ref{theorem:1}]
Kohlberg and Mertens (1986) provided an equivalent representation to games.
Let $u : A \to \dR^I$ be a payoff function.
For every $i \in I$ define two functions $\widetilde u_i : A \to \dR$ and $\overline u_i : A_i \to \dR$ by
\begin{eqnarray}
\overline u_i(a_i) &:=& \frac{1}{|A_{-i}|} \sum_{a_{-i} \in A_{-i}} u_i(a_i,a_{-i}),\\
\widetilde u_i(a) &:=& u_i(a) -\overline u_i(a_i).
\end{eqnarray}
We denote this representation by $u = \langle \widetilde u,\overline u \rangle$.
Since $u_i(a) = \widetilde u_i(a) + \overline u_i(a_i)$,
this representation is one-to-one and onto.

Fix $n > 0$ and define a function $z_n : M_n\to \dR^{\cup_{i \in I} A_i}$ by
\begin{equation*}
z_{n,i,a_i}(u,x) := u_i(a_i,x_{-i}) + \frac{\exp(nu_i(a_i,x_{-i}))}{\sum_{j \in I}\exp(nu_j(a_j,x_{-j}))}, \ \ \ \forall i\in I, a_i\in A_i.
\end{equation*}
Define now a function $\varphi_n : M_n \to \dR^{A \times I}$ by
\begin{equation}
\label{equ:62}
\varphi_n(u,x) := \langle \widetilde u, z_n(u,x) \rangle.
\end{equation}
Lemma~\ref{lemma:g} implies that the function $\varphi_n$ is one-to-one, onto, and an immersion.
The result follows.
\end{proof}

\bigskip

We now prove that the inverse of $g^{(n)}$ converges uniformly as $n$ goes to infinity,
and we provide an explicit form to the limit function, which is nothing but the homeomorphism defined by Kohlberg and Mertens (1986).

\begin{lemma}
\label{lemma:g2}
For every $n > 0$ let $h^{(n)} : \dR^d \to \dR^d$ be the inverse of $g^{(n)}$.
Let $h : \dR^d \to \dR^d$ be the function defined by
\[ h_i(y) := \min\{y_i,\alpha^*\}, \ \ \ \forall i=1,2,\cdots,d, \]
where $\alpha^* := \max\left\{ \alpha \in \dR \colon \sum_{i=1}^d(y_i-\alpha)_+ = 1\right\}$.
Then the sequence of functions $(h^{(n)})_{n > 0}$ converges uniformly to the function $h$.
\end{lemma}

\begin{proof}
Fix $\ep > 0$, and let $n > 0$ be sufficiently large so that $\ep > 1/(1+\exp(\ep n))$.
Fix $y \in \dR^d$ and define $x := h(y)$ and $x^{(n)} := h^{(n)}(y)$.
Assume w.l.o.g.~that $y_1 \leq y_2 \leq \cdots \leq y_d$.
By the definition of $g^{(n)}$ we have $x^{(n)}_1 \leq x^{(n)}_2 \leq\cdots \leq x^{(n)}_d$.
By the definition of $h$ we have $x_1 \leq x_2 \leq \cdots \leq x_d$.
Since
\[ \sum_{i=1}^d (y_i - \alpha^*)_+ = 1 = \sum_{i=1}^d (y_i-x^{(n)}_i) = \sum_{i=1}^d (y_i-x^{(n)}_i)_+, \]
and since $x^{(n)}_1 \leq x^{(n)}_2 \leq\cdots \leq x^{(n)}_d$,
it follows that $x^{(n)}_d \geq \alpha^* = x_d$.

For every $i \in \{1,2,\ldots,d\}$ denote
\[ \alpha_i := y_i - x_i \geq 0, \]
and
\[ \alpha^{(n)}_i := y_i - x^{(n)}_i \geq 0. \]
We now claim that $\alpha_i^{(n)} < \alpha_i + \ep$.
Indeed, assume to the contrary that for some $i \in \{1,2,\ldots,d\}$ we have $\alpha_i^{(n)} \geq \alpha_i + \ep$.
Then in particular
\[ x^{(n)}_i = y_i - \alpha^{(n)}_i \leq y_i - \alpha_i - \ep = x_i - \ep \leq x_d - \ep \leq x_d^{(n)}-\ep. \]
Therefore, by the definition of $g^{(n)}$,
\begin{eqnarray*}
\ep &\leq& \alpha_i^{(n)}
= \frac{\exp(nx_i^{(n)})}{\sum_{j=1}^d\exp(nx_j^{(n)})}\\
&\leq& \frac{\exp(nx_i^{(n)})}{\exp(nx_i^{(n)} + nx_d^{(n)})}\\
&=& \frac{1}{1+\exp\bigl( n(x^{(n)}_d - x^{(n)}_i)\bigr)}
\leq \frac{1}{1+\exp(\ep n)},
\end{eqnarray*}
a contradiction to the choice of $n$.
Since $\sum_{i=1}^d \alpha_i^{(n)} = 1 = \sum_{i=1}^d \alpha_i$,
we deduce that for every $i \in \{1,2,\ldots,d\}$ we have
\[ \alpha_i -d\ep < \alpha_i^{(n)} < \alpha_i + \ep, \]
which implies that $\|h^{(n)}(y) - h(y)\|_\infty \leq d\ep$, and the desired result follows.
\end{proof}

\bigskip

\begin{proof}[Proof of Theorem~\ref{theorem:2}]
In the proof of Kohlberg and Mertens (1986, Theorem~1) it was shown that the following function $\varphi : M \to \dR^{A \times I}$ is a homeomorphism:
\[ \varphi(u,x) := \langle \widetilde u,z(u,x)\rangle, \ \ \ \forall (u,x) \in M, \]
where notations follow the proof of Theorem~\ref{theorem:1} and
\begin{equation*}
z_{i,a_i}(u,x) := u_i(a_i,x_{-i}) + x_i(a_i), \ \ \ \forall i \in I, a_i \in A_i.
\end{equation*}
Theorem~\ref{theorem:2} follows from Lemma~\ref{lemma:g2}.
\end{proof}


\begin{thebibliography}{ZZ}
\bibitem{KohlbergMertens}
Kohlberg E. and Mertens J.-F. (1986)
On the Strategic Stability of Equilibria.
\emph{Econometrica}, \textbf{54}(5), 1003--1037.

\bibitem{MP}
McKelvey R.D. and Palfrey T.R. (1995) Quantal Response Equilibria for Normal Form Games.
\emph{Games and Economic Behavior}, \textbf{10}(1), 6--38.

\bibitem{R1994}
{Ritzberger K. (1994)
The Theory of Normal Form Games from the Differentiable Viewpoint,
\emph{International Journal of Game Theory}, \textbf{23}, 207--236.}

\bibitem{SolanSolan18}
Solan E. and Solan O.N. (2018)
Sunspot Equilibrium in General Quitting Games.
arXiv:1803.00878.

\end{thebibliography}
\end{document}